\documentclass[11pt]{amsart}

\usepackage{
    amsmath,
    amsfonts,
    amssymb,
    amsthm,
    amscd,
    booktabs,
    comment,
    enumitem,
    etoolbox,
    gensymb,    
    mathtools,
    mathdots
}
\usepackage[usenames,dvipsnames]{xcolor}
\usepackage[all]{xy}

\makeatletter
\@namedef{subjclassname@2020}{%
  \textup{2020} Mathematics Subject Classification}
\makeatother

\usepackage[T1]{fontenc}
\usepackage{bbm}                     
\usepackage[colorlinks=true, linkcolor=blue, citecolor=blue, urlcolor=blue, breaklinks=true]{hyperref}

\DeclareFontFamily{OT1}{pzc}{}
\DeclareFontShape{OT1}{pzc}{m}{it}{<-> s * [1.10] pzcmi7t}{}
\DeclareMathAlphabet{\mathpzc}{OT1}{pzc}{m}{it}

\usepackage[capitalize]{cleveref}   

\crefname{defin}{Definition}{Definitions}
\crefname{eg}{Example}{Examples}
\crefname{lem}{Lemma}{Lemmas}
\crefname{theo}{Theorem}{Theorems}
\crefname{equation}{}{}
\crefname{enumi}{}{}

\newcommand\N{\mathbb{N}}

\newcommand\Z{\mathbb{Z}}
\newcommand\kk{\Bbbk}
\newcommand\one{\mathbbm{1}}

\newcommand\ba{\mathbf{a}}

\newcommand\bb{\mathbf{b}}
\newcommand\B{\mathbf{B}}

\newcommand\cC{\mathcal{C}}

\newcommand\fS{\mathfrak{S}}            

\newcommand\op{\mathrm{op}}
\newcommand\rev{\mathrm{rev}}
\newcommand\triv{\mathrm{triv}}

\newcommand\go{\mathsf{I}}              

\newcommand\even{{\bar{0}}}
\newcommand\odd{{\bar{1}}}

\newcommand{\trp}{\varepsilon}          

\newcommand\nilCox{\mathpzc{nilCox}}        
\newcommand\ONC{\mathpzc{ONC}}              
\newcommand\nilHecke{\mathpzc{nilHecke}}    
\newcommand\ONH{\mathpzc{ONH}}              
\newcommand\Pol{\mathpzc{Pol}}              
\newcommand\tower{\mathpzc{tower}}          

\newcommand\Cl{\mathrm{Cl}}                 
\newcommand\nilCoxalg{\mathrm{nilCox}}      
\newcommand\ONCalg{\mathrm{ONC}}            
\newcommand\nilHeckealg{\mathrm{nilHecke}}  
\newcommand\ONHalg{\mathrm{ONH}}            
\newcommand\Polalg{\mathrm{Pol}}            

\DeclareMathOperator{\End}{End}

\DeclareMathOperator{\Hom}{Hom}

\DeclareMathOperator{\id}{id}

\DeclareMathOperator{\tr}{tr}

\usepackage{tikz}
\usetikzlibrary{arrows,patterns}
\usepackage{tikz-cd}

\newcommand{\dotlabel}[1]{$\scriptstyle{#1}$}
\newcommand{\braidup}{to[out=up,in=down]}

\newcommand{\opendot}[1]{\filldraw[fill=white,draw=red] (#1) circle (2pt)}

\newcommand{\token}[3]{
    \filldraw[blue] (#1) circle (1.5pt) node[anchor=#2] {\dotlabel{#3}}
}
\newcommand{\cltoken}[1]{
    \filldraw[blue] (#1) circle (1.5pt)
}
\newcommand{\xtoken}[1]{
    \filldraw[green] (#1) circle (1.5pt)
}
\newcommand\teleport[2]{
    \draw[blue] (#1) -- (#2);
    \filldraw[blue] (#1) circle (1.5pt);
    \filldraw[blue] (#2) circle (1.5pt);
}
\newcommand\xteleport[2]{
    \draw[green] (#1) -- (#2);
    \filldraw[green] (#1) circle (1.5pt);
    \filldraw[green] (#2) circle (1.5pt);
}

\newcommand{\crossgen}{
    \begin{tikzpicture}[centerzero]
        \draw (0.2,-0.2) -- (-0.2,0.2);
        \draw (-0.2,-0.2) -- (0.2,0.2);
    \end{tikzpicture}
}
\newcommand{\dotstrand}{
    \begin{tikzpicture}[centerzero]
        \draw (0,-0.2) -- (0,0.2);
        \opendot{0,0};
    \end{tikzpicture}
}
\newcommand{\tokstrand}[1][a]{
    \begin{tikzpicture}[centerzero]
        \draw (0,-0.2) -- (0,0.2);
        \token{0,0}{west}{#1};
    \end{tikzpicture}
}
\newcommand{\cltokstrand}{
    \begin{tikzpicture}[centerzero]
        \draw (0,-0.2) -- (0,0.2);
        \cltoken{0,0};
    \end{tikzpicture}
}
\newcommand{\xtokstrand}[1][a]{
    \begin{tikzpicture}[centerzero]
        \draw (0,-0.2) -- (0,0.2);
        \xtoken{0,0};
    \end{tikzpicture}
}

\tikzset{anchorbase/.style={>=To,baseline={([yshift=-0.5ex]current bounding box.center)}}}
\tikzset{ 
    centerzero/.style={>=To,baseline={([yshift=-0.5ex](#1))}},
    centerzero/.default={0,0}
}
\tikzset{wipe/.style={white,line width=4pt}}

\newtheorem{theo}{Theorem}[section]

\newtheorem{prop}[theo]{Proposition}
\newtheorem{lem}[theo]{Lemma}
\newtheorem{cor}[theo]{Corollary}

\theoremstyle{definition}
\newtheorem{defin}[theo]{Definition}
\newtheorem{rem}[theo]{Remark}
\newtheorem{egs}[theo]{Examples}

\numberwithin{equation}{section}
\allowdisplaybreaks

\setenumerate[1]{label=(\alph*)}          

\newtoggle{comments}
\newtoggle{details}
\newtoggle{detailsnote}

\iftoggle{comments}{%
  \newcommand{\acomments}[1]{
    \ \\
    {\color{red}
      \textbf{AS:} #1
    }
    \ \\
    }
  \newcommand{\question}[1]{
    \ \\
    {\color{blue}
      \textbf{Question:} #1
    }
    \ \\
    }
}{%
  \newcommand{\acomments}[1]{}
  \newcommand{\question}[1]{}
}

\iftoggle{details}{%
  \newcommand{\details}[1]{
      \ \\
      {\color{OliveGreen}
        \textbf{Details:} #1
      }
      \\
  }
}{%
  \newcommand{\details}[1]{}
}

\begin{document}

\title{Frobenius nilHecke algebras}

\author{Alistair Savage}
\address[A.S.]{
  Department of Mathematics and Statistics \\
  University of Ottawa \\
  Ottawa, ON K1N 6N5, Canada
}
\urladdr{\href{https://alistairsavage.ca}{alistairsavage.ca}, \textrm{\textit{ORCiD}:} \href{https://orcid.org/0000-0002-2859-0239}{orcid.org/0000-0002-2859-0239}}
\email{alistair.savage@uottawa.ca}

\author{John Stuart}
\address[J.S.]{
  Department of Mathematics and Statistics \\
  University of Ottawa \\
  Ottawa, ON K1N 6N5, Canada
}
\email{jstua022@uottawa.ca}

\begin{abstract}
    To any Frobenius superalgebra $A$ we associate towers of \emph{Frobenius nilCoxeter algebras} and \emph{Frobenius nilHecke algebras}.  These act naturally, via \emph{Frobenius divided difference operators}, on \emph{Frobenius polynomial algebras}.  When $A$ is the ground ring, our algebras recover the classical nilCoxeter and nilHecke algebras.  When $A$ is the two-dimensional Clifford algebra, they are Morita equivalent to the odd nilCoxeter and odd nilHecke algebras.
\end{abstract}

\subjclass[2020]{20C08}

\keywords{nilCoxeter algebra, nilHecke algebra, Frobenius algebra, divided difference operator, Demazure operator}

\ifboolexpr{togl{comments} or togl{details}}{%
  {\color{magenta}DETAILS OR COMMENTS ON}
}{%
}

\maketitle
\thispagestyle{empty}

\tableofcontents

\section{Introduction}

NilHecke algebras play an important role in representation theory, geometry, and categorification.  In particular, they appear as algebras of divided difference operators on polynomial rings, as algebras of push-pull operators on the cohomology of the flag variety, and as the simplest case of the quiver Hecke algebras categorifying quantized enveloping algebras.

Loosely speaking, nilHecke algebras are ``nil'' versions of affine Hecke algebras.  Recently, affine Hecke algebras and their degenerate analogues have been generalized in \cite{RS19,Sav20} by incorporating a Frobenius superalgebra $A$ in such a way that when $A$ is the ground ring $\kk$, one recovers the classical constructions.  This general approach allows one to simultaneously handle many variants of (degenerate) affine Hecke algebras, since they occur as special cases of the general construction.  Examples include wreath Hecke algebras, affine Sergeev algebras, affine Yokonuma--Hecke algebras, and affine zigzag algebras.

The goal of the current paper is to continue the program of ``Frobenization'' by defining an analogous Frobenius superalgebra generalization of nilHecke algebras.  Precisely, to every Frobenius superalgebra $A$ and positive integer $n$, we associate a \emph{Frobenius nilCoxeter algebra} $\nilCoxalg_n(A)$ and a \emph{Frobenius nilHecke algebra} $\nilHeckealg_n(A)$.  When $A$ is the ground ring $\kk$, these recover the classical nilCoxeter and nilHecke algebras.  When $A$ is the rank two Clifford algebra, they are Morita equivalent to the odd nilCoxeter and odd nilHecke algebras studied in \cite{EKL14}.  The latter also appear as special cases of the quiver Hecke superalgebras introduced in \cite{KKT16}, and they are nil versions of the degenerate spin affine Hecke algebras defined in \cite[\S3.3]{Wan09}.  When $A$ is the group algebra of a finite cyclic group, $\nilHeckealg_n(A)$ is the nil Yokonuma--Hecke algebra.  For other choices of $A$, including the important cases of zigzag algebras and group algebras of arbitrary finite groups, the Frobenius nilCoxeter algebra and the Frobenius nilHecke algebra do not seem to have appeared before in the literature.

Just as the classical nilHecke algebra acts on the polynomial ring $\kk[x_1,\dotsc,x_n]$ via divided difference operators (also called Demazure operators), the more general Frobenius nilHecke algebra acts on the \emph{Frobenius polynomial algebra} $\Polalg_n(A)$ via \emph{Frobenius divided difference operators}.  In fact, we will see (\cref{BT}) that we have an isomorphism of $\kk$-modules
\[
    \nilHeckealg_n(A) \cong \Polalg_n(A) \otimes \nilCoxalg_n(\kk),
\]
with the two factors on the right being subalgebras.  (Recall that $\nilCoxalg_n(\kk)$ is the usual nilCoxeter algebra.)

Since the Frobenius nilCoxeter and Frobenius nilHecke algebras form \emph{towers of superalgebras}, we use the theory of strict monoidal supercategories to give an efficient presentation of the entire tower all at once.  In particular, we define strict monoidal $\kk$-linear supercategories $\nilCox(A)$ and $\nilHecke(A)$ whose endomorphism algebras are the superalgebras $\nilCoxalg_n(A)$ and $\nilHeckealg_n(A)$.  We then deduce presentations of these as superalgebras.

Throughout the paper we make the assumption that the Frobenius superalgebra $A$ is \emph{symmetric} since this greatly simplifies the exposition and holds in all of our examples of interest.  Here it is important that we allow for \emph{odd} trace maps, since the Clifford superalgebra is only a \emph{symmetric} Frobenius superalgebra when it is equipped with its odd trace map.  In \cref{sec:nonsym} we outline the changes that need to be made in the definitions in order to handle the more general case where $A$ is not necessarily symmetric.

We expect that further study of the algebras introduced in the current paper should lead to natural generalizations of existing results.  In particular, it would be interesting to study cyclotomic quotients, connections to geometry of flag varieties, and relations to categorification of quantized enveloping algebras.

\section{Monoidal supercategories and towers of algebras}

Throughout the paper, we fix a commutative ground ring $\kk$.  All tensor products are over $\kk$ unless otherwise specified.  All superalgebras are associative superalgebras over $\kk$ and all (super)categories are $\kk$-linear.   For a homogeneous element $a$ in a vector superspace, we let $\bar{a} \in \Z_2$ denote its parity.  When we write an equation involving parities of elements, we implicitly assume these elements are homogeneous; we then extend by linearity.

For superalgebras $A = A_\even \oplus A_\odd$ and $B = B_\even \oplus B_\odd$, multiplication in the superalgebra $A \otimes B$ is defined by
\begin{equation}
    (a' \otimes b) (a \otimes b') = (-1)^{\bar a \bar b} a'a \otimes bb'
\end{equation}
for homogeneous $a,a' \in A$, $b,b' \in B$.

Throughout this paper we will work with \emph{strict monoidal supercategories}, in the sense of \cite{BE17}.  We refer the reader to \cite[\S 2]{BSW-foundations} for a summary of this topic well adapted to the current work, or to \cite{BE17} for a thorough treatment.  We summarize here a few crucial properties that play an important role in the present paper.

A \emph{supercategory} means a category enriched in the category of vector superspaces with parity-preserving morphisms.  Thus, its morphism spaces are vector superspaces and composition is parity-preserving.  In a \emph{strict monoidal supercategory}, morphisms satisfy the \emph{super interchange law}:
\begin{equation}\label{interchange}
    (f' \otimes g) \circ (f \otimes g')
    = (-1)^{\bar f \bar g} (f' \circ f) \otimes (g \circ g').
\end{equation}
We denote the unit object by $\one$ and the identity morphism of an object $X$ by $1_X$.  We will use the usual calculus of string diagrams, representing the horizontal composition $f \otimes g$ (resp.\ vertical composition $f \circ g$) of morphisms $f$ and $g$ diagrammatically by drawing $f$ to the left of $g$ (resp.\ drawing $f$ above $g$).  Care is needed with horizontal levels in such diagrams due to the signs arising from the super interchange law:
\begin{equation}\label{intlaw}
    \begin{tikzpicture}[anchorbase]
        \draw (-0.5,-0.5) -- (-0.5,0.5);
        \draw (0.5,-0.5) -- (0.5,0.5);
        \filldraw[fill=white,draw=black] (-0.5,0.15) circle (5pt);
        \filldraw[fill=white,draw=black] (0.5,-0.15) circle (5pt);
        \node at (-0.5,0.15) {$\scriptstyle{f}$};
        \node at (0.5,-0.15) {$\scriptstyle{g}$};
    \end{tikzpicture}
    \quad=\quad
    \begin{tikzpicture}[anchorbase]
        \draw (-0.5,-0.5) -- (-0.5,0.5);
        \draw (0.5,-0.5) -- (0.5,0.5);
        \filldraw[fill=white,draw=black] (-0.5,0) circle (5pt);
        \filldraw[fill=white,draw=black] (0.5,0) circle (5pt);
        \node at (-0.5,0) {$\scriptstyle{f}$};
        \node at (0.5,0) {$\scriptstyle{g}$};
    \end{tikzpicture}
    \quad=\quad
    (-1)^{\bar f\bar g}\
    \begin{tikzpicture}[anchorbase]
        \draw (-0.5,-0.5) -- (-0.5,0.5);
        \draw (0.5,-0.5) -- (0.5,0.5);
        \filldraw[fill=white,draw=black] (-0.5,-0.15) circle (5pt);
        \filldraw[fill=white,draw=black] (0.5,0.15) circle (5pt);
        \node at (-0.5,-0.15) {$\scriptstyle{f}$};
        \node at (0.5,0.15) {$\scriptstyle{g}$};
    \end{tikzpicture}
    \ .
\end{equation}

In fact, all of the categories we consider in the current paper will be strict monoidal supercategories generated by a single object $\go$.  The domain and codomain of a string diagram can then be read from the number of strands at the bottom and top, respectively, of the diagram.  For example,
\[
    \crossgen \colon \go^{\otimes 2} \to \go^{\otimes 2}
    \qquad \text{and} \qquad
    \dotstrand \colon \go \to \go.
\]
(In fact, all of our generating morphisms will be endomorphisms, having the same domain and codomain.)  For this reason, we will often omit the domain and codomain when introducing the generating morphisms of a category.  When numbering strands, we will always number \emph{from right to left}.

A strict monoidal supercategory $\cC$ with one generating object $\go$ gives rise to a \emph{tower of algebras}
\[
    \End_\cC(\go^{\otimes n}),\ n \in \N.
\]
We use this idea to introduce various families of algebras in an extremely efficient way, giving a presentation of $\cC$ with a small number of generating morphisms and relations.  If we then wish to have a presentation of the endomorphism algebras of $\cC$ \emph{as superalgebras}, we use the following result.

\begin{prop} \label{zebra}
    Suppose $\cC$ is a strict monoidal supercategory with one generating object $\go$ and generating morphisms $f_i \in \End_\cC(\go^{\otimes n_i})$, $i \in I$, subject to the relations $R_j \in \End_\cC(\go^{\otimes m_j})$, $j \in J$.  Then $\End_\cC(\go^{\otimes n})$ is generated as an algebra by the elements
    \[
        1_\one^{n-n_i-k} \otimes f_i \otimes 1_\one^k,\quad i \in I,\ 0 \le k \le n-n_i,
    \]
    subject to the relations
    \begin{equation} \label{zebra1}
        1_\one^{n-m_j-k} \otimes R_j \otimes 1_\one^k,\quad j \in J,\ 0 \le k \le n-m_j,
    \end{equation}
    and the relations
    \begin{equation} \label{zebra2}
        \left( 1_\one^{\otimes k_1} \otimes f_i \otimes 1^{\otimes k_2} \right)
        \left( 1_\one^{\otimes l_1} \otimes f_j \otimes 1^{\otimes l_2} \right)
        - (-1)^{\overline{f_i} \overline{f_j}}
        \left( 1_\one^{\otimes l_1} \otimes f_j \otimes 1^{\otimes l_2} \right)
        \left( 1_\one^{\otimes k_1} \otimes f_i \otimes 1^{\otimes k_2} \right)
    \end{equation}
    for $i,j \in I$, $k_1+n_i+k_2 = n = l_1 + n_j + l_2$, $k_2 \ge n_j + l_2$.
\end{prop}

\begin{proof}
    This follows from Theorems~5.2 and~5.4 of \cite{Liu18}.  Note that the relations \cref{zebra2} correspond to the super interchange law.
\end{proof}

Note that the assumptions in \cref{zebra} are quite strong: all generating morphisms and relations are endomorphisms.  All of our categories will have this property.

\section{Frobenius superalgebras}

We fix a symmetric Frobenius superalgebra $A$ with trace map $\tr \colon A \to \kk$ of parity $\trp$.  By definition, this means that $\tr$ is a homogeneous $\kk$-linear map of parity $\trp$ satisfying
\begin{equation} \label{jackal}
    \tr(ab) = (-1)^{\bar{a} \bar{b}} \tr(b a),\quad a,b \in A,
\end{equation}
and $A$ has a basis $\B_A$ with a dual basis $\B_A^\vee = \{b^\vee : b \in \B_A\}$ satisfying
\begin{equation} \label{lake}
    \tr(a^\vee b) = \delta_{a,b},\quad a,b \in \B_A.
\end{equation}
It follows that
\begin{equation} \label{beam}
    \sum_{b \in \B_A} \tr(b^\vee a) b
    = a
    = \sum_{b \in \B_A} \tr(ab) b^\vee,\quad a \in A.
\end{equation}
Note that $\bar{b} + \overline{b^\vee} = \trp$.  By abuse of notation, we will often refer to $A$ itself as a Frobenius superalgebra, leaving the trace map $\tr$ implicit.

\begin{egs}
    It may be useful for the reader to keep in mind the following important examples of Frobenius superalgebras:
    \begin{enumerate}
        \item $A = \kk$ with $\tr = \id$;
        \item $A = \Cl$, the rank two Clifford superalgebra, with odd trace map (see \cref{sec:Clifford});
        \item $A = \kk G$, the group algebra of a finite group $G$, with trace map given by projection onto the identity element of $G$;
        \item $A$ a zigzag algebra in the sense of \cite{HK01}.
    \end{enumerate}
\end{egs}

 The symmetric group $\fS_n$ acts on $A^{\otimes n}$ by superpermutations.  In particular, the simple transposition $s_i$ acts by
\begin{equation} \label{swap}
    s_i ( a_n \otimes \dotsb \otimes a_1 )
    = (-1)^{\bar{a}_i \bar{a}_{i+1}} a_n \otimes \dotsb \otimes a_{i+2} \otimes a_i \otimes a_{i+1} \otimes a_{i-1} \otimes \dotsb a_1,\quad a_1,\dotsc,a_n \in A,
\end{equation}
extended by linearity.  Note that here, and throughout the paper, we number factors from \emph{right to left}.

Define
\begin{equation} \label{tau}
    \tau := \sum_{b \in \B_A} (-1)^{\trp \bar{b}} b \otimes b^\vee
    \in A^{\otimes 2}.
\end{equation}
The element $\tau$ has parity $\trp$ and is independent of the chosen basis $\B_A$.

\begin{lem}
    We have
    \begin{gather} \label{doubledual}
        (b^\vee)^\vee = (-1)^{\bar{b} + \trp \bar{b}} b,\quad b \in \B_A,
        \\ \label{train}
        \ba \tau = (-1)^{\trp \bar{a}} \tau s_1(\ba),\quad
        \ba \in A^{\otimes 2}.
    \end{gather}
\end{lem}

\begin{proof}
    The equation \cref{doubledual} follows immediately from \cref{lake}.  To prove \cref{train}, it suffices to consider $\ba$ of the form $a \otimes 1$ and $1 \otimes a$, $a \in A$.  We compute
    \[
        (a \otimes 1) \tau
        \overset{\cref{beam}}{=} \sum_{b,c \in \B_A} (-1)^{\trp \bar{b}} \tr(c^\vee ab) c \otimes b^\vee
        = \sum_{b,c \in \B_A} (-1)^{\trp(\bar{a} + \bar{c})} c \otimes \tr(c^\vee ab) b^\vee
        \overset{\cref{beam}}{=} (-1)^{\trp \bar{a}} \tau (1 \otimes a),
    \]
    where, in the second equality, we used the fact that $\tr(c^\vee a b)=0$ unless $\bar{c} + \bar{a} + \bar{b} = \even$.  Similarly,
    \[
        (1 \otimes a) \tau
        \overset{\cref{beam}}{=} \sum_{b,c \in \B_A} (-1)^{\trp \bar{b} + \bar{a} \bar{b}} b \otimes \tr(a b^\vee c) c^\vee
        \overset{\cref{jackal}}{=} \sum_{b,c \in \B_A} (-1)^{\trp \bar{c} + \bar{a} \bar{c}} \tr(b^\vee c a) b \otimes c^\vee
        \overset{\cref{beam}}{=} (-1)^{\trp \bar{a}} \tau (a \otimes 1).
    \]
\end{proof}

\begin{defin}
    The \emph{Frobenius tower category} $\tower(A)$ is the strict monoidal supercategory with one generating object $\go$, generating morphisms (called \emph{tokens})
    \[
        \tokstrand,\ a \in A,
    \]
    and relations
    \begin{equation} \label{tokrel}
        \begin{tikzpicture}[anchorbase]
            \draw (0,0) -- (0,0.7);
            \token{0,0.35}{west}{1};
        \end{tikzpicture}
        =
        \begin{tikzpicture}[anchorbase]
            \draw (0,0) -- (0,0.7);
        \end{tikzpicture}
        \ ,\quad
        \lambda\
        \begin{tikzpicture}[anchorbase]
            \draw (0,0) -- (0,0.7);
            \token{0,0.35}{west}{a};
        \end{tikzpicture}
        + \mu\
        \begin{tikzpicture}[anchorbase]
            \draw (0,0) -- (0,0.7);
            \token{0,0.35}{west}{b};
        \end{tikzpicture}
        =
        \begin{tikzpicture}[anchorbase]
            \draw (0,0) -- (0,0.7);
            \token{0,0.35}{west}{\lambda a + \mu b};
        \end{tikzpicture}
        ,\quad
        \begin{tikzpicture}[anchorbase]
            \draw (0,0) -- (0,0.7);
            \token{0,0.2}{east}{b};
            \token{0,0.45}{east}{a};
        \end{tikzpicture}
        =
        \begin{tikzpicture}[anchorbase]
            \draw (0,0) -- (0,0.7);
            \token{0,0.35}{west}{ab};
        \end{tikzpicture}
        \ ,\quad a,b \in A,\ \lambda,\mu \in \kk.
    \end{equation}
    We declare the parity of $\tokstrand$ to be the same as that of $a$.
\end{defin}

Note that relations \cref{tokrel} are precisely the relations we need in order to have a homomorphism of superalgebras
\begin{equation}
    A \to \End_{\tower(A)}(\go),\quad
    a \mapsto \tokstrand,\ a \in A.
\end{equation}
In fact, $\tower(A)$ is the free monoidal supercategory generated by an object with endomorphism superalgebra $A$.  It follows from \cref{zebra} that we have an isomorphism
\begin{equation} \label{minotaur}
    A^{\otimes n} \xrightarrow{\cong} \End_{\tower(A)}(\go^{\otimes n}),
\end{equation}
sending $1^{\otimes (n-i)} \otimes a \otimes 1^{\otimes (i-1)}$, $a \in A$, to a token labeled $a$ on the $i$-th strand.  As always, we label strands from \emph{right to left}.

Following \cite[\S4]{BSW-foundations}, we introduce the \emph{teleporters}
\begin{equation}\label{brexit}
    \begin{tikzpicture}[anchorbase]
        \draw[->] (0,-0.4) --(0,0.4);
        \draw[->] (0.5,-0.4) -- (0.5,0.4);
        \teleport{0,0}{0.5,0};
    \end{tikzpicture}
    := \sum_{b \in \B_A} (-1)^{\trp \bar{b}}
    \begin{tikzpicture}[anchorbase]
        \draw[->] (0,-0.4) --(0,0.4);
        \draw[->] (0.5,-0.4) -- (0.5,0.4);
        \token{0,0}{east}{b};
        \token{0.5,0}{west}{b^\vee};
    \end{tikzpicture}
    = \sum_{b \in \B_A} (-1)^{\trp \bar{b}}
    \begin{tikzpicture}[anchorbase]
        \draw[->] (0,-0.4) --(0,0.4);
        \draw[->] (0.5,-0.4) -- (0.5,0.4);
        \token{0,0.15}{east}{b};
        \token{0.5,-0.15}{west}{b^\vee};
    \end{tikzpicture}
    = (-1)^\trp \sum_{b \in \B_A} (-1)^{\trp \bar{b}}
    \begin{tikzpicture}[anchorbase]
        \draw[->] (0,-0.4) --(0,0.4);
        \draw[->] (0.5,-0.4) -- (0.5,0.4);
        \token{0,-0.15}{east}{b^\vee};
        \token{0.5,0.15}{west}{b};
    \end{tikzpicture}
    \ ,
\end{equation}
where, in the last equality, we used \cref{doubledual} and changed basis in the sum.  The teleporter is independent of the chosen basis $\B_A$ and has parity $\trp$.  It is the element of $\tower(A)$ corresponding to $\tau$ (see \cref{tau}) under the isomorphism of \cref{minotaur}.  It follows from \cref{train} that tokens can ``teleport'' across teleporters (justifying the terminology) in the sense that, for $a \in A$, we have
\begin{equation} \label{teleport}
    \begin{tikzpicture}[anchorbase]
        \draw[->] (0,-0.5) --(0,0.5);
        \draw[->] (0.5,-0.5) -- (0.5,0.5);
        \token{0,0.25}{east}{a};
        \teleport{0,0}{0.5,0};
    \end{tikzpicture}
    = (-1)^{\trp \bar{a}}\
    \begin{tikzpicture}[anchorbase]
        \draw[->] (0,-0.5) --(0,0.5);
        \draw[->] (0.5,-0.5) -- (0.5,0.5);
        \token{0.5,-0.25}{west}{a};
        \teleport{0,0}{0.5,0};
    \end{tikzpicture}
    \ ,\qquad
    \begin{tikzpicture}[anchorbase]
        \draw[->] (0,-0.5) --(0,0.5);
        \draw[->] (0.5,-0.5) -- (0.5,0.5);
        \token{0,-0.25}{east}{a};
        \teleport{0,0}{0.5,0};
    \end{tikzpicture}
    = (-1)^{\trp \bar{a}}\
    \begin{tikzpicture}[anchorbase]
        \draw[->] (0,-0.5) --(0,0.5);
        \draw[->] (0.5,-0.5) -- (0.5,0.5);
        \token{0.5,0.25}{west}{a};
        \teleport{0,0}{0.5,0};
    \end{tikzpicture}
    \ .
\end{equation}

\section{Frobenius nilCoxeter algebras}

In this section we introduce the tower of Frobenius nilCoxeter algebras.

\begin{defin}
    The \emph{Frobenius nilCoxeter category} $\nilCox(A)$ is the strict monoidal supercategory with one generating object $\go$ and generating morphisms
    \[
        \crossgen,\quad \tokstrand,\ a \in A,
    \]
    subject to the relations \cref{tokrel} and
    \begin{equation} \label{NCrel}
        \begin{tikzpicture}[anchorbase]
            \draw (0.2,-0.5) \braidup (-0.2,0) \braidup (0.2,0.5);
            \draw (-0.2,-0.5) \braidup (0.2,0) \braidup (-0.2,0.5);
        \end{tikzpicture}
        \ =\
        0
        \ ,\quad
        \begin{tikzpicture}[anchorbase]
            \draw (0.4,-0.5) -- (-0.4,0.5);
            \draw (0,-0.5) \braidup (-0.4,0) \braidup (0,0.5);
            \draw (-0.4,-0.5) -- (0.4,0.5);
        \end{tikzpicture}
        \ =\
        \begin{tikzpicture}[anchorbase]
            \draw (0.4,-0.5) -- (-0.4,0.5);
            \draw (0,-0.5) \braidup (0.4,0) \braidup (0,0.5);
            \draw (-0.4,-0.5) -- (0.4,0.5);
        \end{tikzpicture}
        \ ,\quad
        \begin{tikzpicture}[centerzero]
            \draw (0.3,-0.4) -- (-0.3,0.4);
            \draw (-0.3,-0.4) -- (0.3,0.4);
            \token{-0.15,-0.2}{east}{a};
        \end{tikzpicture}
        \ =\
        \begin{tikzpicture}[centerzero]
            \draw (0.3,-0.4) -- (-0.3,0.4);
            \draw (-0.3,-0.4) -- (0.3,0.4);
            \token{0.15,0.2}{west}{a};
        \end{tikzpicture}
        \ ,\quad
        \begin{tikzpicture}[centerzero]
            \draw (-0.3,-0.4) -- (0.3,0.4);
            \draw (0.3,-0.4) -- (-0.3,0.4);
            \token{0.15,-0.2}{west}{a};
        \end{tikzpicture}
        \ =\
        \begin{tikzpicture}[centerzero]
            \draw (-0.3,-0.4) -- (0.3,0.4);
            \draw (0.3,-0.4) -- (-0.3,0.4);
            \token{-0.15,0.2}{east}{a};
        \end{tikzpicture}
        \ .
    \end{equation}
    We refer to the generator $\crossgen$ as a \emph{crossing} and declare it to be even.  The parity of the token $\tokstrand$ is the same as the parity of $a$.  For $n \in \N$, we define the \emph{Frobenius nilCoxeter algebra}
    \[
        \nilCoxalg_n(A) := \End_{\nilCox(A)}(\go^{\otimes n}).
    \]
\end{defin}

\begin{prop} \label{unicorn}
    As a superalgebra, $\nilCoxalg_n(A)$ is isomorphic to the free product of $A^{\otimes n}$ and the free associative superalgebra on even generators $u_1,\dotsc,u_{n-1}$, subject to the relations
    \begin{align}
        u_i^2 &= 0, & 1 \le i \le n-1, \label{unicorn1} \\
        u_i u_j &= u_j u_i, & |i-j| > 1, \label{unicorn2} \\
        u_i u_{i+1} u_i &= u_{i+1} u_i u_{i+1}, & 1 \le i \le n-2, \label{unicorn3} \\
        u_i \ba &= s_i(\ba) u_i,& 1 \le i \le n-1,\ \ba \in A^{\otimes n}. \label{unicorn4}
    \end{align}
    Under this isomorphism, $u_i$ corresponds to the crossing of strands $i$ and $i+1$, while $1^{\otimes (n-i)} \otimes a \otimes 1^{\otimes (i-1)}$ corresponds to a token labeled $a$ on strand $i$.  As usual, we number strands from right to left.
\end{prop}

\begin{proof}
    This follows from \cref{zebra}.
\end{proof}

In what follows, we will identify $\nilCoxalg_n(A)$ with the algebra presented as in \cref{unicorn}.

\begin{rem}
    Note that, in fact, the definition of $\nilCox(A)$, and hence $\nilCoxalg_n(A)$, only involves the superalgebra structure on $A$, and not the trace map.  Thus these are defined for any superalgebra $A$.  The trace map will be important later.
\end{rem}

For $w \in \fS_n$, define
\[
    u_w := u_{i_1} u_{i_2} \dotsm u_{i_k},
\]
where $w = s_{i_1} s_{i_2} \dotsm s_{i_k}$ is a reduced expression for $w$.  By \cref{unicorn2,unicorn3}, the element $u_w$ is independent of the choice of reduced expression.

When $A = \kk$, the algebra $\nilCoxalg_n(\kk)$ is purely even and is the usual nilCoxeter algebra.  It is well known that $\nilCoxalg_n(\kk)$ has basis given by the $u_w$, $w \in \fS_n$, with multiplication given by
\[
    u_w u_v =
    \begin{cases}
        u_{wv} & \text{if } \ell(wv) = \ell(w) + \ell(v), \\
        0 & \text{otherwise},
    \end{cases}
\]
where $\ell(w)$ is the length of the element $w \in \fS_n$.  For general $A$, we have an isomorphism of $\kk$-modules
\begin{equation}
    \nilCoxalg_n(A) \cong A^{\otimes n} \otimes \nilCoxalg_n(\kk),
\end{equation}
and the two factors are subalgebras.

\section{Frobenius nilHecke algebras}

We now introduce the tower of Frobenius nilHecke algebras.

\begin{defin}
    The \emph{Frobenius nilHecke category} $\nilHecke(A)$ is the strict monoidal supercategory with one generating object $\go$ and generating morphisms
    \[
        \crossgen,\quad \dotstrand,\quad \tokstrand,\ a \in A,
    \]
    subject to the relations \cref{tokrel}, \cref{NCrel}, and
    \begin{equation} \label{dotslide}
        \begin{tikzpicture}[centerzero]
            \draw (0.3,-0.3) -- (-0.3,0.3);
            \draw (-0.3,-0.3) -- (0.3,0.3);
            \opendot{-0.15,-0.15};
        \end{tikzpicture}
        \ -\
        \begin{tikzpicture}[centerzero]
            \draw (0.3,-0.3) -- (-0.3,0.3);
            \draw (-0.3,-0.3) -- (0.3,0.3);
            \opendot{0.15,0.15};
        \end{tikzpicture}
        \ =\
        \begin{tikzpicture}[centerzero]
            \draw (-0.2,-0.3) -- (-0.2,0.3);
            \draw (0.2,-0.3) -- (0.2,0.3);
            \teleport{-0.2,0}{0.2,0};
        \end{tikzpicture}
        \ ,\quad
        \begin{tikzpicture}[centerzero]
            \draw (-0.3,-0.3) -- (0.3,0.3);
            \draw (0.3,-0.3) -- (-0.3,0.3);
            \opendot{-0.15,0.15};
        \end{tikzpicture}
        \ -\
        \begin{tikzpicture}[centerzero]
            \draw (-0.3,-0.3) -- (0.3,0.3);
            \draw (0.3,-0.3) -- (-0.3,0.3);
            \opendot{0.15,-0.15};
        \end{tikzpicture}
        \ = (-1)^\trp\
        \begin{tikzpicture}[centerzero]
            \draw (-0.2,-0.3) -- (-0.2,0.3);
            \draw (0.2,-0.3) -- (0.2,0.3);
            \teleport{-0.2,0}{0.2,0};
        \end{tikzpicture}
        \ ,\quad
        \begin{tikzpicture}[centerzero]
            \draw (0,-0.3) -- (0,0.3);
            \opendot{0,-0.13};
            \token{0,0.13}{east}{a};
        \end{tikzpicture}
        \ = (-1)^{\trp \bar{a}}\
        \begin{tikzpicture}[centerzero]
            \draw (0,-0.3) -- (0,0.3);
            \opendot{0,0.13};
            \token{0,-0.13}{west}{a};
        \end{tikzpicture}
        \ ,\ a \in A.
    \end{equation}
    We refer to the generator $\dotstrand$ as a \emph{dot} and declare it to be of parity $\trp$.  The crossings are even and the parity of the token $\tokstrand$ is the parity of $a$.  For $n \in \N$, we define the \emph{Frobenius nilHecke algebra}
    \[
        \nilHeckealg_n(A) := \End_{\nilHecke(A)}(\go^{\otimes n}).
    \]
\end{defin}

For $1 \le i \le n-1$, define
\begin{equation} \label{taui}
    \tau_i := 1^{\otimes (n-i-1)} \otimes \tau \otimes 1^{\otimes (i-1)} \in A^{\otimes n},
\end{equation}
where $\tau$ is defined as in \cref{tau}.

\begin{prop} \label{pegasus}
    As a superalgebra, $\nilHeckealg_n(A)$ is isomorphic to the free product of $\nilCox(A)$ and the free associative superalgebra on generators $x_1,\dotsc,x_n$ of parity $\trp$, subject to the relations
    \begin{align}
        x_i x_j &= (-1)^\trp x_j x_i, & i \ne j, \label{NHxx} \\
        \ba x_i &= (-1)^{\trp \bar{\ba}} x_i \ba,& \ba \in A^{\otimes n},\ 1 \le i \le n, \label{NHax} \\
        x_j u_i &= u_i x_j, & j \notin \{i,i+1\}, \\
        u_i x_{i+1} &= x_i u_i + \tau_i, & 1 \le i \le n-1, \\
        x_{i+1} u_i &= u_i x_i + (-1)^\trp \tau_i, & 1 \le i \le n-1.
    \end{align}
\end{prop}

\begin{proof}
    This follows from \cref{zebra}.
\end{proof}

We will give an explicit basis of $\nilHeckealg_n(A)$ in the next section (\cref{basis}) after we describe a natural action on the Frobenius polynomial algebra.

\begin{egs}
    \begin{enumerate}
        \item When $A = \kk$, we have $\tau_i = 1$, for all $1 \le i \le n-1$, and $\trp = \even$.  It follows that $\nilHeckealg_n(\kk)$, which is purely even, is the usual nilHecke algebra.

        \item When $A = \Cl$ is the rank two Clifford superalgebra, we will see in \cref{sec:Clifford} that $\nilHeckealg_n(\Cl)$ is Morita equivalent to the odd nilHecke algebra of \cite{EKL14}.

        \item When $A$ is the group algebra of a finite cyclic group, $\nilHeckealg_n(A)$ is the nil Yokonuma--Hecke algebra; see \cite[\S 3]{Cui16}.
    \end{enumerate}
\end{egs}

\begin{prop}
    Up to isomorphism, the category $\nilHecke(A)$ depends only on the underlying superalgebra $A$, and not on the trace map.  Hence the same is true of the algebras $\nilHeckealg_n(A)$.
\end{prop}

\begin{proof}
    Let $\tr_1$ and $\tr_2$ be two trace maps on $A$.  We use the notation $\nilHecke(A,\tr_i)$, instead of $\nilHecke(A)$, to denote the nilHecke category corresponding to the Frobenius algebra $A$ with trace map $\tr_i$ for $i \in \{1,2\}$. Then there exists a homogeneous invertible element $u \in A$ such that $\tr_2(a) = \tr_1(au)$ for all $a \in A$.  (For a proof of this fact in the setting of superalgebras, see \cite[Prop.~3.4]{PS16}.)  Then it is straightforward to verify that we have an isomorphism of strict monoidal supercategories
    \begin{equation}
        \nilHecke(A,\tr_1) \to \nilHecke(A,\tr_2),\quad
        \crossgen \mapsto \crossgen\ ,\quad
        \dotstrand \mapsto (-1)^{\bar{u}}\
        \begin{tikzpicture}[centerzero]
            \draw (0,-0.3) -- (0,0.3);
            \opendot{0,0.13};
            \token{0,-0.13}{west}{u};
        \end{tikzpicture}
        \ ,\quad
        \tokstrand \mapsto \tokstrand,\ a \in A.
        \qedhere
    \end{equation}
\end{proof}

For the usual nilHecke algebra (i.e.\ $A = \kk$), one often considers the grading given by declaring the $x_i$ to be of degree $2$ and the $u_i$ to be of degree $-2$.  This can be generalized to the setting of Frobenius nilHecke algebras as follows.  Suppose that $A$ is a $\Z$-graded Frobenius superalgebra and that the trace map is homogeneous of $\Z$-degree $-d$ and parity $\trp$.  It follows that $\tau$ has $\Z$-degree $d$.  Then $\nilHecke(A)$ is a strict graded monoidal supercategory where we define the $\Z$-degrees of the generators to be
\begin{equation}
    \deg(\crossgen) = d-2,\quad
    \deg(\dotstrand) = d+2,\quad
    \deg(\tokstrand) = \deg(a),\ a \in A.
\end{equation}
It follows that the Frobenius nilHecke algebras $\nilHeckealg_n(A)$ are $\Z$-graded, with
\begin{equation}
    \deg(u_i) = d-2,\quad
    \deg(x_j) = d+2,
\end{equation}
and where the degree of $\ba$, considered an element of $\nilHeckealg_n(A)$, is the same as its degree as an element of $A^{\otimes n}$.

There are two symmetries of $\nilHecke(A)$ that induce symmetries of the Frobenius nilHecke algebras.  First, we have an isomorphism of monoidal supercategories
\begin{equation}
    \Omega_\leftrightarrow \colon \nilHecke(A) \to \nilHecke(A)^\rev,\qquad
    \crossgen \mapsto - \crossgen,\quad
    \dotstrand \mapsto \dotstrand,\quad
    \tokstrand \mapsto \tokstrand,\ a \in A.
\end{equation}
Here $\nilHecke(A)^\rev$ denotes the reversed supercategory, with the same objects, morphisms $f^\rev \colon X \to Y$ for $f \in \Hom_{\nilHecke(A)}(X,Y)$ (and the map $f \mapsto f^\rev$ is $\kk$-linear), composition $f^\rev \circ g^\rev = (f \circ g)^\rev$, and tensor product $f^\rev \otimes g^\rev = (g \otimes f)^\rev$.  On a string diagram representing a morphism in $\nilHecke(A)$, $\Omega_\leftrightarrow$ is given by reflecting diagrams in a vertical line and multiplying by $(-1)^k$, where $k$ is the number of crossings in the diagram.  This functor induces isomorphisms of superalgebras
\begin{equation}
    \omega_\leftrightarrow \colon \nilHeckealg_n(A) \to \nilHeckealg_n(A),\quad
    u_i \mapsto - u_{n-i},\quad
    x_i \mapsto x_{n+1-i},\quad
    \ba \mapsto \pi (\ba),
\end{equation}
where $\pi$ is the longest element of $\fS_n$.

The second symmetry of $\nilHecke(A)$ is the isomorphism of monoidal supercategories
\begin{equation}
    \Omega_\updownarrow \colon \nilHecke(A) \to \nilHecke(A)^\op,\qquad
    \crossgen \mapsto (-1)^\trp\ \crossgen,\quad
    \dotstrand \mapsto \dotstrand,\quad
    \tokstrand \mapsto \tokstrand,\ a \in A.
\end{equation}
Here $\nilHecke(A)^\op$ denotes the opposite supercategory, with the same objects, morphisms $f^\op \colon Y \to X$ for $f \in \Hom_{\nilHecke(A)}(X,Y)$ (and the map $f \mapsto f^\op$ is $\kk$-linear), composition $f^\op \circ g^\op = (-1)^{\bar{f} \bar{g}} (g \circ f)^\op$, and tensor product $f^\op \otimes g^\op = (f \otimes g)^\op$.  On a string diagram representing a morphism in $\nilHecke(A)$, $\Omega_\updownarrow$ is given by reflecting diagrams in a horizontal line and multiplying by $(-1)^{k \trp}$, where $k$ is the number of crossings in the diagram.  This functor induces isomorphisms of superalgebras
\begin{equation}
    \omega_\updownarrow \colon \nilHeckealg_n(A) \to \nilHeckealg_n(A)^\op,\quad
    u_i \mapsto (-1)^\trp u_i,\quad
    x_i \mapsto x_i,\quad
    \ba \mapsto \ba.
\end{equation}

\section{Frobenius polynomial algebras\label{sec:Pol}}

In this section, we describe how the Frobenius nilHecke algebra acts on the Frobenius polynomial algebra via Frobenius divided difference operators.  This allows us to prove a basis theorem for the Frobenius nilHecke algebra.

\begin{defin}
    We define the \emph{Frobenius polynomial category} $\Pol(A)$ to be the strict monoidal supercategory with one generating object $\go$ and generating morphisms
    \[
        \dotstrand,\quad \tokstrand,\ a \in A,
    \]
    subject to the relations \cref{tokrel} and
    \begin{equation}
        \begin{tikzpicture}[centerzero]
            \draw (0,-0.3) -- (0,0.3);
            \opendot{0,-0.13};
            \token{0,0.13}{east}{a};
        \end{tikzpicture}
        \ = (-1)^{\trp \bar{a}}\
        \begin{tikzpicture}[centerzero]
            \draw (0,-0.3) -- (0,0.3);
            \opendot{0,0.13};
            \token{0,-0.13}{west}{a};
        \end{tikzpicture}
        \ ,\ a \in A.
    \end{equation}
    The dot $\dotstrand$ has parity $\trp$ and the parity of the token $\tokstrand$ is the parity of $a$.  For $n \in \N$, we define the \emph{Frobenius polynomial algebra}
    \[
        \Polalg_n(A) := \End_{\Pol(A)}(\go^{\otimes n}).
    \]
\end{defin}

\begin{prop}
    As a superalgebra, $\Polalg_n(A)$ is isomorphic to the free product of $A^{\otimes n}$ and the free associative superalgebra on the generators $x_1,\dotsc,x_n$ of parity $\trp$, subject to the relations \cref{NHxx,NHax}.
\end{prop}

\begin{proof}
    This follows from \cref{zebra}.
\end{proof}

The algebra $\Polalg_n(\kk)$ has basis
\[
    x_1^{k_1} x_2^{k_2} \dotsm x_n^{k_n},\quad k_1,k_2,\dotsc,k_n \in \N.
\]
We have an isomorphism of $\kk$-modules
\begin{equation}
    \Polalg_n(A) \cong A^{\otimes n} \otimes \Polalg_n(\kk),
\end{equation}
and the two factors are subalgebras.  When $\trp=\even$, the algebra $\Polalg_n(\kk) = \kk[x_1,\dotsc,x_n]$ is the usual polynomial algebra.

The action \cref{swap} of the symmetric group $\fS_n$ on $A^{\otimes n}$ extends naturally to an action on $\Polalg_n(A)$ via superalgebra isomorphisms, where
\begin{equation}
    s_i(x_j) = x_{s_i(j)},\quad 1 \le j \le n,\ 1 \le i \le n-1.
\end{equation}

We define the \emph{Frobenius divided difference operators} (or \emph{Frobenius Demazure operators}) to be the linear operators $\partial_i \colon \Polalg_n(A) \to \Polalg_n(A)$, $1 \le i \le n-1$, defined inductively by
\begin{equation} \label{horse}
    \partial_i(\ba) = 0,\quad
    \partial_i(x_j) =
    \begin{cases}
        (-1)^{\trp+1} \tau_i & \text{if } j = i, \\
        \tau_i & \text{if } j = i+1, \\
        0 & \text{otherwise},
    \end{cases}
\end{equation}
for $\ba \in A^{\otimes n}$, and the twisted Leibniz rule
\begin{equation}
    \partial_i(fg) = \partial_i(f) g + s_i(f) \partial_i(g),\quad f,g \in \Polalg_n(A).
\end{equation}

\begin{lem} \label{baboon}
    If $\trp = \even$, then, for $1 \le i \le n-1$, we have
    \begin{equation} \label{mule}
        \partial_i(\ba p) = \tau_i \ba \frac{p - s_i(p)}{x_{i+1}-x_i},
    \end{equation}
    for all $\ba \in A^{\otimes n}$ and $p \in \kk[x_1,\dotsc,x_n]$.  In particular, if $A = \kk$, then $\partial_i$ is the usual divided difference operator.
\end{lem}

\begin{proof}
    For the sake of the proof, let $\partial_i'$ denote the operator $\partial_i$ of \cref{mule}.  It is clear that $\partial_i'(\ba) = \partial_i(\ba)$ for $\ba \in A^{\otimes n}$ and $\partial_i'(x_j) = \partial_i(x_j)$ for all $j$.  Now, for $\ba, \bb \in A^{\otimes n}$ and $p,q \in \kk[x_1,\dotsc,x_n]$, we have
    \begin{multline*}
        \partial_i'(\ba p \bb q)
        = \partial_i'(\ba \bb p q)
        = \tau_i \ba \bb \frac{pq - s_i(pq)}{x_{i+1}-x_i}
        = \tau_i \ba \bb \left( \frac{p - s_i(p)}{x_{i+1}-x_i} q + s_i(p) \frac{q-s_i(q)}{x_{i+1}-x_i} \right)
        \\
        \overset{\cref{train}}{=} \partial_i'(\ba p) \bb q + s_i(\ba p) \partial_i'(\bb q).
    \end{multline*}
    Hence $\partial_i'$ also satisfies the twisted Leibniz rule, and hence agrees with $\partial_i$.
\end{proof}

We have natural maps $\Polalg_n(A) \to \nilHeckealg_n(A)$ and $\nilCoxalg_n(\kk) \to \nilHeckealg_n(A)$ mapping $\ba \in A^{\otimes n}$, $x_i$, and $u_j$ to the elements of $\nilHeckealg_n(A)$ denoted by the same symbols.  Abusing notation, we will view elements of $\Polalg_n(A)$ and $\nilCoxalg_n(\kk)$ as elements of $\nilHeckealg_n(A)$ via their images under these natural maps.  (It will follow from \cref{BT} that these maps are injective, but we do not use that fact yet.)

\begin{theo} \label{BT}
    The map
    \begin{equation} \label{platypus}
        \Polalg_n(A) \otimes \nilCoxalg_n(\kk) \to \nilHeckealg_n(A),\quad
        f \otimes z \mapsto fz,
    \end{equation}
    extended by $\kk$-linearity, is an isomorphism of $\kk$-modules.
\end{theo}

\begin{proof}
    The proof uses standard techniques, and so we only sketch the argument here.  Let $\psi$ be the map \cref{platypus}.  First, one notes that, using the defining relations of $\nilHeckealg_n(A)$, elements of $\Polalg_n(A)$ can be moved to the left of elements of $\nilCoxalg_n(A)$ modulo terms of lower degree in the $x_i$.  This shows that $\psi$ is surjective.

    Next, it is a straightforward, but lengthy, verification to check that $\nilHeckealg_n(A)$ acts on $\Polalg_n(A) \otimes \nilCoxalg_n(\kk)$ by
    \begin{equation} \label{bill}
        g \cdot (f \otimes z) = gf \otimes z, \qquad
        u_i \cdot (f \otimes z) = s_i(f) \otimes u_i z + \partial_i(f) \otimes z,
    \end{equation}
    for $f,g \in \Polalg_n(A)$, $z \in \nilCoxalg_n(\kk)$, and $1 \le i \le n-1$.  Then, since
    \[
        \psi(f \otimes u_w) \cdot (1 \otimes 1) = f \otimes u_w,\quad
        f \in \Polalg_n(A),\ w \in \fS_n,
    \]
    the map $\psi$ is also injective.
\end{proof}

\begin{cor} \label{basis}
    The Frobenius nilHecke algebra $\nilHeckealg_n(A)$ has basis
    \[
        \ba x_1^{k_1} \dotsm x_n^{k_n} u_w,\quad
        \ba \in \B_A^{\otimes n},\ k_1,\dotsc,k_n \in \N,\ w \in \fS_n.
    \]
    (Recall that $\B_A$ is a basis of $A$.)
\end{cor}

In light of \cref{BT}, we may view $A^{\otimes n}$, $\Polalg_n(A)$, and $\nilCoxalg_n(A)$ as subalgebras of $\nilHecke(A)$.

\begin{cor}
    In $\nilHeckealg_n(A)$, we have
    \[
        u_i f = s_i(f) u_i + \partial_i(f),\quad
        1 \le i \le n-1,\ f \in \Polalg_n(A).
    \]
\end{cor}

\begin{proof}
    This follows from \cref{bill}.
\end{proof}

\begin{prop}
    There is an action of $\nilHeckealg_n(A)$ on $\Polalg_n(A)$ given by
    \begin{equation} \label{cardinal}
        x_i \cdot f = x_i f,\quad
        u_j \cdot f = \partial_j(f),\quad
        1 \le i \le n,\ 1 \le j \le n-1,\ f \in \Polalg_n(A).
    \end{equation}
\end{prop}

\begin{proof}
    This can be proved in two ways.  The first is to check directly, via a straightforward computation, that this action satisfies the defining relations of $\nilHeckealg_n(A)$.  The second is to construct this action as an induced module.  Let $\triv$ denote the trivial rank one $\nilCoxalg_n(\kk)$-module on which $u_i$ acts as zero for all $1 \le i \le n-1$.  In then follows from \cref{BT} that we have an isomorphism of $\kk$-modules
    \[
        \nilHeckealg_n(A) \otimes_{\nilCox_n(\kk)} \triv \cong \Polalg_n(A).
    \]
    Under this isomorphism, the action \cref{bill} corresponds precisely to the action \cref{cardinal}.
\end{proof}

We refer to the action \cref{cardinal} as the \emph{polynomial representation} of $\nilHeckealg_n(A)$.  When $A = \kk$, it corresponds to the well-known action of the nilCoxeter algebra on the polynomial ring $\kk[x_1,\dotsc,x_n]$ via divided difference operators.

\begin{rem}
    The polynomial representation of the usual nilCoxeter algebra $\nilCoxalg_n(\kk)$ is faithful.  However, this is \emph{not} true for general $A$.  As an example, consider the case where $A = \kk[y]/y^2 \kk[y]$ is purely even, with trace map given by $\tr(\lambda + \mu y) = \mu$ for $\lambda,\mu \in \kk$.  Then $A$ has basis $\B_A = \{1,y\}$ with dual basis given by $1^\vee = y$, $y^\vee = 1$.  The element $(y \otimes y) u_1 \in \nilHeckealg_2(A)$ is nonzero by \cref{BT}.  However, for $\ba \in A^{\otimes n}$ and $f \in \kk[x_1,\dotsc,x_n]$, we have
    \[
        (y \otimes y) u_1  \cdot (\ba f)
        = (y \otimes y) \partial_i(\ba f)
        \overset{\cref{mule}}{=} (y \otimes y) (1 \otimes y + y \otimes 1) \ba \frac{f - s_1(f)}{x_2-x_1}
        = 0.
    \]
    Thus $(y \otimes y)u_1$ acts as zero.
\end{rem}

\section{The Clifford case and odd nilHecke algebras\label{sec:Clifford}}

In this section we consider the special case when $A = \Cl$ is the rank 2 Clifford superalgebra.  By definition, $\Cl$ is the superalgebra generated by a single odd generator $c$ subject to the relation $c^2=1$.  It follows that $\Cl$ is free with basis $\B_\Cl = \{1,c\}$.  We fix the \emph{odd} trace map
\[
    \tr \colon \Cl \to \kk,\quad \tr(1) = 0,\ \tr(c) = 1.
\]
Hence $\trp = \odd$.  It is straightforward to verify that $\Cl$ is a symmetric Frobenius superalgebra with this trace map, with dual basis given by
\[
    1^\vee = c,\qquad c^\vee = 1.
\]
Thus, for $1 \le i \le n-1$, we have
\begin{equation}
    \tau_i = c_i - c_{i+1},\quad \text{where }
    c_i := 1^{\otimes (n-i)} \otimes c \otimes 1^{\otimes (i-1)} \in \Cl^{\otimes n}.
\end{equation}
(Note that $\Cl$ is \emph{not} symmetric under the even trace map $1 \mapsto 1$, $c \mapsto 0$.)

Consider the \emph{Clifford nilCoxeter category} $\nilCox(\Cl)$ and the \emph{Clifford nilHecke category} $\nilHecke(\Cl)$, which are the Frobenius nilCoxeter category and the Frobenius nilHecke category, respectively, for the special choice of $A = \Cl$.  We define
\[
    \cltokstrand := \tokstrand[c].
\]
Thus, for example, we have
\[
    \begin{tikzpicture}[centerzero]
        \draw (0,-0.3) -- (0,0.3);
        \cltoken{0,-0.13};
        \cltoken{0,0.13};
    \end{tikzpicture}
    \ =\
    \begin{tikzpicture}[centerzero]
        \draw (0,-0.3) -- (0,0.3);
    \end{tikzpicture}
    \qquad \text{and} \qquad
    \begin{tikzpicture}[centerzero]
        \draw (-0.2,-0.2) -- (-0.2,0.2);
        \draw (0.2,-0.2) -- (0.2,0.2);
        \teleport{-0.2,0}{0.2,0};
    \end{tikzpicture}
    \ =\
    \begin{tikzpicture}[centerzero]
        \draw (-0.2,-0.2) -- (-0.2,0.2);
        \draw (0.2,-0.2) -- (0.2,0.2);
        \cltoken{0.2,0};
    \end{tikzpicture}
    \ -\
    \begin{tikzpicture}[centerzero]
        \draw (-0.2,-0.2) -- (-0.2,0.2);
        \draw (0.2,-0.2) -- (0.2,0.2);
        \cltoken{-0.2,0};
    \end{tikzpicture}
    \ .
\]
We call $\nilCoxalg_n(\Cl)$ the \emph{Clifford nilCoxeter algebra} and we call $\nilHeckealg_n(\Cl)$ the \emph{Clifford nilHecke algebra}.  The goal of this section is to compare these to the odd nilCoxeter algebra and odd nilHecke algebra.

\begin{defin}
    The \emph{odd nilCoxeter category} $\ONC$ is the strict monoidal supercategory with one generating object $\go$ and one odd generating morphism
    \[
        \crossgen,
    \]
    subject to the relations
    \begin{equation} \label{oddNC}
        \begin{tikzpicture}[anchorbase]
            \draw (0.2,-0.5) \braidup (-0.2,0) \braidup (0.2,0.5);
            \draw (-0.2,-0.5) \braidup (0.2,0) \braidup (-0.2,0.5);
        \end{tikzpicture}
        \ =\
        0
        \ ,\quad
        \begin{tikzpicture}[anchorbase]
            \draw (0.4,-0.5) -- (-0.4,0.5);
            \draw (0,-0.5) \braidup (-0.4,0) \braidup (0,0.5);
            \draw (-0.4,-0.5) -- (0.4,0.5);
        \end{tikzpicture}
        \ =\
        \begin{tikzpicture}[anchorbase]
            \draw (0.4,-0.5) -- (-0.4,0.5);
            \draw (0,-0.5) \braidup (0.4,0) \braidup (0,0.5);
            \draw (-0.4,-0.5) -- (0.4,0.5);
        \end{tikzpicture}
        \ .
    \end{equation}
    For $n \in \N$, the \emph{odd nilCoxeter algebra} is the superalgebra
    \[
        \ONCalg_n := \End_{\ONC}(\go^{\otimes n}).
    \]
\end{defin}

\begin{defin}[{cf.\ \cite[\S3]{EKL14}}]
    The \emph{odd nilHecke category} $\ONH$ is the strict monoidal supercategory with one generating object $\go$ and odd generating morphisms
    \[
        \dotstrand,\quad \crossgen,
    \]
    subject to the relations \cref{oddNC} and
    \begin{equation} \label{fox}
        \begin{tikzpicture}[centerzero]
            \draw (0.3,-0.3) -- (-0.3,0.3);
            \draw (-0.3,-0.3) -- (0.3,0.3);
            \opendot{-0.15,-0.15};
        \end{tikzpicture}
        \ +\
        \begin{tikzpicture}[centerzero]
            \draw (0.3,-0.3) -- (-0.3,0.3);
            \draw (-0.3,-0.3) -- (0.3,0.3);
            \opendot{0.15,0.15};
        \end{tikzpicture}
        \ =\
        \begin{tikzpicture}[centerzero]
            \draw (-0.2,-0.3) -- (-0.2,0.3);
            \draw (0.2,-0.3) -- (0.2,0.3);
        \end{tikzpicture}
        \ =\
        \begin{tikzpicture}[centerzero]
            \draw (-0.3,-0.3) -- (0.3,0.3);
            \draw (0.3,-0.3) -- (-0.3,0.3);
            \opendot{-0.15,0.15};
        \end{tikzpicture}
        \ +\
        \begin{tikzpicture}[centerzero]
            \draw (-0.3,-0.3) -- (0.3,0.3);
            \draw (0.3,-0.3) -- (-0.3,0.3);
            \opendot{0.15,-0.15};
        \end{tikzpicture}
        \ .
    \end{equation}
    For $n \in \N$, the \emph{odd nilHecke algebra} is the superalgebra
    \[
        \ONHalg_n := \End_{\ONH}(\go^{\otimes n}).
    \]
\end{defin}

The following proposition shows that the above definition of the odd nilHecke algebra agrees with the definition of \cite[\S2.2]{EKL14}.

\begin{prop} \label{griffon}
    As a superalgebra, $\ONCalg_n$ is isomorphic to the superalgebra with odd generators $v_1,\dotsc,v_{n-1}$, subject to the relations
    \begin{align}
        v_i^2 &=0, & 1 \le i \le n-1, \label{ONC1} \\
        v_i v_j &= - v_j v_i, & |i-j| > 1, \label{ONC2} \\
        v_i v_{i+1} v_i &= v_{i+1} v_i v_{i+1}, & 1 \le i \le n-2, \label{ONC3}
    \end{align}
    As a superalgebra, $\ONHalg_n$ is isomorphic to the superalgebra with odd generators $y_1,\dotsc,y_n,v_1,\dotsc,v_{n-1}$, subject to the relations \cref{ONC1,ONC2,ONC3} and
    \begin{align}
        y_i y_j &= - y_j y_i, & i \ne j, \\
        y_j v_i &= - v_i y_j, & i \notin \{j,j+1\}, \\
        v_i y_{i+1} + y_i v_i &= 1, & 1 \le i \le n-2, \\
        y_{i+1} v_i + v_i x_i &= 1, & 1 \le i \le n-2.
    \end{align}
\end{prop}

Let $\ONC(\Cl)$ be the strict monoidal category obtained from $\ONC$ by adjoining an extra odd morphism
\begin{equation} \label{greendot}
    \xtokstrand
\end{equation}
subject to the relations
\begin{equation} \label{greenslide}
    \begin{tikzpicture}[centerzero]
        \draw (0,-0.3) -- (0,0.3);
        \xtoken{0,-0.13};
        \xtoken{0,0.13};
    \end{tikzpicture}
    \ =\
    \begin{tikzpicture}[centerzero]
        \draw (0,-0.3) -- (0,0.3);
    \end{tikzpicture}
    \ ,\quad
    \begin{tikzpicture}[centerzero]
        \draw (-0.3,-0.3) -- (0.3,0.3);
        \draw (0.3,-0.3) -- (-0.3,0.3);
        \xtoken{-0.15,-0.15};
    \end{tikzpicture}
    \ = -\
    \begin{tikzpicture}[centerzero]
        \draw (-0.3,-0.3) -- (0.3,0.3);
        \draw (0.3,-0.3) -- (-0.3,0.3);
        \xtoken{-0.15,0.15};
    \end{tikzpicture}
    \ ,\quad
    \begin{tikzpicture}[centerzero]
        \draw (-0.3,-0.3) -- (0.3,0.3);
        \draw (0.3,-0.3) -- (-0.3,0.3);
        \xtoken{0.15,-0.15};
    \end{tikzpicture}
    \ = -\
    \begin{tikzpicture}[centerzero]
        \draw (-0.3,-0.3) -- (0.3,0.3);
        \draw (0.3,-0.3) -- (-0.3,0.3);
        \xtoken{0.15,0.15};
    \end{tikzpicture}
    \ .
\end{equation}
Let $\ONH(\Cl)$ be the strict monoidal category obtained from $\ONH$ by adjoint one extra odd morphism \cref{greendot} subject to the relations \cref{greenslide} and
\begin{equation} \label{hare}
    \begin{tikzpicture}[centerzero]
        \draw (0,-0.3) -- (0,0.3);
        \opendot{0,0.13};
        \xtoken{0,-0.13};
    \end{tikzpicture}
    \ = -\
    \begin{tikzpicture}[centerzero]
        \draw (0,-0.3) -- (0,0.3);
        \xtoken{0,0.13};
        \opendot{0,-0.13};
    \end{tikzpicture}
    \ .
\end{equation}
We have colored the extra generator green (instead of blue) to remind the reader that the relation involving this generator and the crossing is quite different from the relation involving usual tokens and the crossing.  We also have the green teleporter
\[
    \begin{tikzpicture}[centerzero]
        \draw (-0.2,-0.2) -- (-0.2,0.2);
        \draw (0.2,-0.2) -- (0.2,0.2);
        \xteleport{-0.2,0}{0.2,0};
    \end{tikzpicture}
    \ :=\
    \begin{tikzpicture}[centerzero]
        \draw (-0.2,-0.2) -- (-0.2,0.2);
        \draw (0.2,-0.2) -- (0.2,0.2);
        \xtoken{0.2,0};
    \end{tikzpicture}
    \ -\
    \begin{tikzpicture}[centerzero]
        \draw (-0.2,-0.2) -- (-0.2,0.2);
        \draw (0.2,-0.2) -- (0.2,0.2);
        \xtoken{-0.2,0};
    \end{tikzpicture}
    \ .
\]

For $n \in \N$, the isomorphisms of \cref{griffon} extend to isomorphisms of superalgebras
\begin{equation}
    \End_{\ONC(\Cl)}(\go^{\otimes n}) \cong \ONCalg_n \otimes \Cl^{\otimes n},\qquad
    \End_{\ONH(\Cl)}(\go^{\otimes n}) \cong \ONHalg_n \otimes \Cl^{\otimes n}.
\end{equation}
with a green token $\xtokstrand$ on the $i$-th strand mapping to the element $c_i$.

\begin{theo} \label{mirror}
    Suppose $2$ is invertible in the ground ring $\kk$.  We have an isomorphism of monoidal supercategories $\nilCox(\Cl) \to \ONC(\Cl)$ given on the generating object by $\go \mapsto \go$ and on the generating morphisms by
    \begin{equation} \label{mirror1}
        \cltokstrand \mapsto \xtokstrand,\quad
        \begin{tikzpicture}[centerzero]
            \draw (-0.3,-0.3) -- (0.3,0.3);
            \draw (0.3,-0.3) -- (-0.3,0.3);
        \end{tikzpicture}
        \mapsto
        \begin{tikzpicture}[centerzero]
            \draw (-0.3,-0.3) -- (0.3,0.3);
            \draw (0.3,-0.3) -- (-0.3,0.3);
            \xteleport{-0.15,0.15}{0.15,0.15};
        \end{tikzpicture}
        \ .
    \end{equation}
    The inverse is given by
    \begin{equation}
        \xtokstrand \mapsto \cltokstrand,\quad
        \begin{tikzpicture}[centerzero]
            \draw (-0.3,-0.3) -- (0.3,0.3);
            \draw (0.3,-0.3) -- (-0.3,0.3);
        \end{tikzpicture}
        \mapsto \frac{1}{2}\
        \begin{tikzpicture}[centerzero]
            \draw (-0.3,-0.3) -- (0.3,0.3);
            \draw (0.3,-0.3) -- (-0.3,0.3);
            \teleport{-0.15,0.15}{0.15,0.15};
        \end{tikzpicture}
        \ .
    \end{equation}
    This extends to an isomorphism of monoidal supercategories $\nilHecke(\Cl) \to \ONH(\Cl)$ by defining it on the dot to be
    \[
        \dotstrand \mapsto \dotstrand\ .
    \]
\end{theo}

\begin{proof}
    It is a straightforward exercise to verify that the given maps respect the defining relations and are mutually inverse.  For example, if $\Psi$ is the map \cref{mirror1}, then we have
    \begin{gather*}
        \Psi
        \left(
            \begin{tikzpicture}[centerzero]
                \draw (0.3,-0.3) -- (-0.3,0.3);
                \draw (-0.3,-0.3) -- (0.3,0.3);
                \opendot{-0.15,-0.15};
            \end{tikzpicture}
            \ -\
            \begin{tikzpicture}[centerzero]
                \draw (0.3,-0.3) -- (-0.3,0.3);
                \draw (-0.3,-0.3) -- (0.3,0.3);
                \opendot{0.15,0.15};
            \end{tikzpicture}
        \right)
        =
        \begin{tikzpicture}[centerzero]
            \draw (-0.4,-0.4) -- (0.4,0.4);
            \draw (0.4,-0.4) -- (-0.4,0.4);
            \xteleport{-0.15,0.15}{0.15,0.15};
            \opendot{-0.2,-0.2};
        \end{tikzpicture}
        -
        \begin{tikzpicture}[centerzero]
            \draw (-0.4,-0.4) -- (0.4,0.4);
            \draw (0.4,-0.4) -- (-0.4,0.4);
            \xteleport{-0.15,0.15}{0.15,0.15};
            \opendot{0.3,0.3};
        \end{tikzpicture}
        \overset{\cref{hare}}{=}
        \begin{tikzpicture}[centerzero]
            \draw (-0.4,-0.4) -- (0.4,0.4);
            \draw (0.4,-0.4) -- (-0.4,0.4);
            \xteleport{-0.15,0.15}{0.15,0.15};
            \opendot{-0.2,-0.2};
        \end{tikzpicture}
        +
        \begin{tikzpicture}[centerzero]
            \draw (-0.4,-0.4) -- (0.4,0.4);
            \draw (0.4,-0.4) -- (-0.4,0.4);
            \xteleport{-0.3,0.3}{0.3,0.3};
            \opendot{0.15,0.15};
        \end{tikzpicture}
        \overset{\cref{fox}}{=}
        \begin{tikzpicture}[centerzero]
            \draw (-0.2,-0.4) -- (-0.2,0.4);
            \draw (0.2,-0.4) -- (0.2,0.4);
            \xteleport{-0.2,0}{0.2,0};
        \end{tikzpicture}
        = \Psi
        \left(
            \begin{tikzpicture}[centerzero]
                \draw (-0.2,-0.3) -- (-0.2,0.3);
                \draw (0.2,-0.3) -- (0.2,0.3);
                \teleport{-0.2,0}{0.2,0};
            \end{tikzpicture}
        \right),
        \\
        \Psi
        \left(
            \begin{tikzpicture}[centerzero]
                \draw (0.3,-0.3) -- (-0.3,0.3);
                \draw (-0.3,-0.3) -- (0.3,0.3);
                \opendot{-0.15,0.15};
            \end{tikzpicture}
            \ -\
            \begin{tikzpicture}[centerzero]
                \draw (0.3,-0.3) -- (-0.3,0.3);
                \draw (-0.3,-0.3) -- (0.3,0.3);
                \opendot{0.15,-0.15};
            \end{tikzpicture}
        \right)
        =
        \begin{tikzpicture}[centerzero]
            \draw (-0.4,-0.4) -- (0.4,0.4);
            \draw (0.4,-0.4) -- (-0.4,0.4);
            \xteleport{-0.15,0.15}{0.15,0.15};
            \opendot{-0.3,0.3};
        \end{tikzpicture}
        -
        \begin{tikzpicture}[centerzero]
            \draw (-0.4,-0.4) -- (0.4,0.4);
            \draw (0.4,-0.4) -- (-0.4,0.4);
            \xteleport{-0.15,0.15}{0.15,0.15};
            \opendot{0.2,-0.2};
        \end{tikzpicture}
        \overset{\cref{hare}}{=} -
        \begin{tikzpicture}[centerzero]
            \draw (-0.4,-0.4) -- (0.4,0.4);
            \draw (0.4,-0.4) -- (-0.4,0.4);
            \xteleport{-0.3,0.3}{0.3,0.3};
            \opendot{-0.15,0.15};
        \end{tikzpicture}
        -
        \begin{tikzpicture}[centerzero]
            \draw (-0.4,-0.4) -- (0.4,0.4);
            \draw (0.4,-0.4) -- (-0.4,0.4);
            \xteleport{-0.15,0.15}{0.15,0.15};
            \opendot{0.3,-0.3};
        \end{tikzpicture}
        \overset{\cref{fox}}{=} -\
        \begin{tikzpicture}[centerzero]
            \draw (-0.2,-0.4) -- (-0.2,0.4);
            \draw (0.2,-0.4) -- (0.2,0.4);
            \xteleport{-0.2,0}{0.2,0};
        \end{tikzpicture}
        = \Psi
        \left(
            (-1)^\trp
            \begin{tikzpicture}[centerzero]
                \draw (-0.2,-0.3) -- (-0.2,0.3);
                \draw (0.2,-0.3) -- (0.2,0.3);
                \teleport{-0.2,0}{0.2,0};
            \end{tikzpicture}
        \right). \qedhere
    \end{gather*}
\end{proof}

\begin{cor} \label{hyena}
    Suppose $2$ is invertible in the ground ring $\kk$.  We have an isomorphism of superalgebras $\nilCoxalg_n(\Cl) \to \ONCalg_n \otimes \Cl^{\otimes n}$ given by
    \[
        c_i \mapsto c_i,\quad
        u_i \mapsto \tau_i v_i = (c_i - c_{i+1}) v_i.
    \]
    This extends to an isomorphism of superalgebras $\nilHeckealg_n(\Cl) \to \ONHalg_n \otimes \Cl^{\otimes n}$ by defining
    \[
        x_i \mapsto y_i.
    \]
    In particular, the Clifford nilCoxeter algebra and the odd nilCoxeter algebra are Morita equivalent, as are the Clifford nilHecke algebra and the odd nilHecke algebra.
\end{cor}

\begin{proof}
    This follows immediately from \cref{mirror} and the fact that $\Cl$ is a simple superalgebra.
\end{proof}

\Cref{hyena} should be viewed as a nil version of the well known fact that the Sergeev algebra is Morita equivalent to the spin symmetric group algebra, and the fact that the degenerate affine Hecke--Clifford algebra (also called the affine Sergeev algebra) is Morita equivalent to the degenerate spin affine Hecke algebra; see \cite[Th.~4.1]{Wan09}.  In the non-degenerate setting, one also has the analogous Morita equivalence between the (affine) Hecke--Clifford algebra and the spin (affine) Hecke algebra; see \cite[Th.~5.1]{Wan07}.

\section{The nonsymmetric case\label{sec:nonsym}}

For simplicity of exposition, we have assumed throughout this paper that the Frobenius superalgebra $A$ is \emph{symmetric}.  Most of the results also hold without this assumption, provided one makes the appropriate modifications.

Let $A$ be a Frobenius superalgebra, not necessarily symmetric, with trace map $\tr$ of parity $\trp$.  Then $A$ has a \emph{Nakayama automorphism} $\psi \colon A \to A$.  This is the superalgebra automorphism of $A$ uniquely determined by the condition
\begin{equation} \label{Nakayama}
    \tr(ab) = (-1)^{\bar{a} \bar{b}} \tr(b \psi(a)),\quad a,b \in A.
\end{equation}
The Frobenius superalgebra $A$ is symmetric if $\psi = \id$.

Define the teleporters
\begin{equation}
    \begin{tikzpicture}[centerzero]
        \draw (-0.2,-0.3) -- (-0.2,0.3);
        \draw (0.2,-0.3) -- (0.2,0.3);
        \teleport{-0.2,0.1}{0.2,-0.1};
    \end{tikzpicture}
    = \sum_{b \in \B_A} (-1)^{\trp \bar{b}}
    \begin{tikzpicture}[centerzero]
        \draw (-0.2,-0.3) -- (-0.2,0.3);
        \draw (0.2,-0.3) -- (0.2,0.3);
        \token{-0.2,0.1}{east}{b};
        \token{0.2,-0.1}{west}{b^\vee};
    \end{tikzpicture}
    \qquad \text{and} \qquad
    \begin{tikzpicture}[centerzero]
        \draw (-0.2,-0.3) -- (-0.2,0.3);
        \draw (0.2,-0.3) -- (0.2,0.3);
        \teleport{-0.2,-0.1}{0.2,0.1};
    \end{tikzpicture}
    = \sum_{b \in \B_A} (-1)^{\trp \bar{b}}
    \begin{tikzpicture}[centerzero]
        \draw (-0.2,-0.3) -- (-0.2,0.3);
        \draw (0.2,-0.3) -- (0.2,0.3);
        \token{-0.2,-0.1}{east}{b^\vee};
        \token{0.2,0.1}{west}{b};
    \end{tikzpicture}
    \ .
\end{equation}
(When $\psi = \id$, these are related by a sign; see~\cref{brexit}.)  Then we have
\begin{align}
    \begin{tikzpicture}[anchorbase]
        \draw[->] (0,-0.5) --(0,0.5);
        \draw[->] (0.5,-0.5) -- (0.5,0.5);
        \token{0,0.3}{east}{a};
        \teleport{0,0.1}{0.5,-0.1};
    \end{tikzpicture}
    &= (-1)^{\trp \bar{a}}\
    \begin{tikzpicture}[anchorbase]
        \draw[->] (0,-0.5) --(0,0.5);
        \draw[->] (0.5,-0.5) -- (0.5,0.5);
        \token{0.5,-0.3}{west}{a};
        \teleport{0,0.1}{0.5,-0.1};
    \end{tikzpicture}
    \ ,&
    \begin{tikzpicture}[anchorbase]
        \draw[->] (0,-0.5) --(0,0.5);
        \draw[->] (0.5,-0.5) -- (0.5,0.5);
        \token{0,-0.3}{east}{\psi(a)};
        \teleport{0,0.1}{0.5,-0.1};
    \end{tikzpicture}
    &= (-1)^{\trp \bar{a}}\
    \begin{tikzpicture}[anchorbase]
        \draw[->] (0,-0.5) --(0,0.5);
        \draw[->] (0.5,-0.5) -- (0.5,0.5);
        \token{0.5,0.3}{west}{a};
        \teleport{0,0.1}{0.5,-0.1};
    \end{tikzpicture}
    \ ,
    \\
    \begin{tikzpicture}[anchorbase]
        \draw[->] (0,-0.5) --(0,0.5);
        \draw[->] (0.5,-0.5) -- (0.5,0.5);
        \token{0,-0.3}{east}{a};
        \teleport{0,-0.1}{0.5,0.1};
    \end{tikzpicture}
    &= (-1)^{\trp \bar{a}}\
    \begin{tikzpicture}[anchorbase]
        \draw[->] (0,-0.5) --(0,0.5);
        \draw[->] (0.5,-0.5) -- (0.5,0.5);
        \token{0.5,0.3}{west}{a};
        \teleport{0,-0.1}{0.5,0.1};
    \end{tikzpicture}
    \ ,&
    \begin{tikzpicture}[anchorbase]
        \draw[->] (0,-0.5) --(0,0.5);
        \draw[->] (0.5,-0.5) -- (0.5,0.5);
        \token{0,0.3}{east}{a};
        \teleport{0,-0.1}{0.5,0.1};
    \end{tikzpicture}
    &= (-1)^{\trp \bar{a}}\
    \begin{tikzpicture}[anchorbase]
        \draw[->] (0,-0.5) --(0,0.5);
        \draw[->] (0.5,-0.5) -- (0.5,0.5);
        \token{0.5,-0.3}{west}{\psi(a)};
        \teleport{0,-0.1}{0.5,0.1};
    \end{tikzpicture}
    \ .
\end{align}
We should then modify the relations \cref{dotslide} in the definition of $\nilHecke(A)$ to be
\begin{equation}
        \begin{tikzpicture}[centerzero]
            \draw (0.3,-0.3) -- (-0.3,0.3);
            \draw (-0.3,-0.3) -- (0.3,0.3);
            \opendot{-0.15,-0.15};
        \end{tikzpicture}
        \ -\
        \begin{tikzpicture}[centerzero]
            \draw (0.3,-0.3) -- (-0.3,0.3);
            \draw (-0.3,-0.3) -- (0.3,0.3);
            \opendot{0.15,0.15};
        \end{tikzpicture}
        \ =\
        \begin{tikzpicture}[centerzero]
            \draw (-0.2,-0.3) -- (-0.2,0.3);
            \draw (0.2,-0.3) -- (0.2,0.3);
            \teleport{-0.2,0.1}{0.2,-0.1};
        \end{tikzpicture}
        \ ,\quad
        \begin{tikzpicture}[centerzero]
            \draw (-0.3,-0.3) -- (0.3,0.3);
            \draw (0.3,-0.3) -- (-0.3,0.3);
            \opendot{-0.15,0.15};
        \end{tikzpicture}
        \ -\
        \begin{tikzpicture}[centerzero]
            \draw (-0.3,-0.3) -- (0.3,0.3);
            \draw (0.3,-0.3) -- (-0.3,0.3);
            \opendot{0.15,-0.15};
        \end{tikzpicture}
        \ =\
        \begin{tikzpicture}[centerzero]
            \draw (-0.2,-0.3) -- (-0.2,0.3);
            \draw (0.2,-0.3) -- (0.2,0.3);
            \teleport{-0.2,-0.1}{0.2,0.1};
        \end{tikzpicture}
        \ ,\quad
        \begin{tikzpicture}[centerzero]
            \draw (0,-0.3) -- (0,0.3);
            \opendot{0,-0.13};
            \token{0,0.13}{east}{a};
        \end{tikzpicture}
        \ = (-1)^{\trp \bar{a}}\
        \begin{tikzpicture}[centerzero]
            \draw (0,-0.3) -- (0,0.3);
            \opendot{0,0.13};
            \token{0,-0.13}{west}{\psi(a)};
        \end{tikzpicture}
        \ ,\ a \in A.
\end{equation}
One then needs to carefully carry the Nakayama automorphism through all the constructions in the paper.

\section*{Acknowledgements}

This research of A.~Savage was supported by Discovery Grant RGPIN-2017-03854 from the Natural Sciences and Engineering Research Council of Canada (NSERC).  J.~Stuart was also supported by this Discovery Grant and an NSERC Undergraduate Student Research Award.


\bibliographystyle{alphaurl}
\bibliography{FNH}

\end{document}